\documentclass{article}

\usepackage[margin = 1.5in]{geometry}

\usepackage{tikz}
\usepackage{tikz-cd}
\usepackage{url, hyperref}

\tikzset{->-/.style={decoration={
  markings,
  mark=at position .45 with {\arrow{>}}},postaction={decorate}}}

\usepackage{amsmath}
\usepackage{amssymb}
\usepackage{enumitem}
\usepackage{graphicx}
\usepackage{mathdots}
\usepackage{color}
\usepackage{diagbox}
\usepackage{array, makecell}
\usepackage{rotating}
\usepackage{amsthm}

\newtheorem{definition}{Definition}
\newtheorem{theorem}[definition]{Theorem}

\newtheorem{proposition}[definition]{Proposition}
\newtheorem{corollary}[definition]{Corollary}
\newtheorem{lemma}[definition]{Lemma}
\newtheorem{remark}[definition]{Remark}

\def\C{\mathbb{C}}
\def\Q{\mathbb{Q}}
\def\Z{\mathbb{Z}}

\def\R{\mathbb{R}}
\def\H{\mathsf{H}}
\def\L{\mathsf{L}}
\def\E{\mathsf{E}}
\def\m{\mathsf{m}}
\def\QH{QH^*(X)}
\def\P{\mathbb{P}}
\def\oM{\overline{\mathcal{M}}_{g,n}}
\def\vT{\mathrm{vTev}^X_{g,n,k}}

\title{Quantum Euler class and virtual Tevelev degrees of Fano complete intersections}
\author{Alessio Cela }
\date{April 2022}

\begin{document}

\maketitle

\begin{abstract}
We compute the quantum Euler class of Fano complete intersections $X$ in a projective space. In particular, we prove a recent conjecture of A. Buch and R. Pandharipande, namely \cite[Conjecture 5.14]{BP}. Finally we apply our result to obtain formulas for the virtual Tevelev degrees of $X$. An algorithm computing all genus $0$ two-point Hyperplane Gromov Witten invariants of $X$ is illustrated along the way.
\end{abstract}

\tableofcontents

\section{Introduction}

\subsection{Quantum Euler class of a variety}

Let X be a nonsingular, projective, algebraic variety over $\C$ of dimension $r$ and let $\{\gamma_j \}_{j=0}^N \subset H^*(X)$ be a homogeneous basis with $\gamma_0=1 $ and $\gamma_N=\mathsf{P}$ the point class. The small quantum cohomology ring $\QH$\footnote{Unless otherwise specified, (co)homology and quantum cohomology will always be taken with $\Q$-coefficients in this paper.} of $X$ is defined via the $3-$point genus $0$ Gromov-Witten invariants:
$$
\gamma_i \star \gamma_j=\sum_{\beta \in H_2(X,\Z)} \sum_k \langle\gamma_i,\gamma_j,\gamma_k^{\vee}\rangle^X_{0,\beta} q^\beta \gamma_k
$$
where $\gamma_k^{\vee} \in H^*(X)$ is the dual of $\gamma_k$ with respect to the intersection form on $X$, defined by the conditions 
$$
\int_X \gamma_j \cup \gamma_k^{\vee}= \delta_{j,k} \ \text{for } j = 0,...,N.
$$
Here we are following the notation of \cite{FP}.

Let also
$$
\Delta= \sum_j \gamma_j^{\vee} \otimes \gamma_j \in H ^*(X) \otimes H^*(X)
$$
be the Künneth decomposition of the diagonal class of $X \times X$. 

The \textbf{quantum Euler class} of $X$ is the image of $\Delta$ under the product map
$$
H^*(X) \otimes H^*(X) \xrightarrow{\star} \QH.
$$

This is a canonically defined element of $QH^*(X)$, first introduced by Abrams in \cite{Abr}. In terms of the basis $\{ \gamma_j \}$, we have
$$
\E= \sum_j \gamma_j^{\vee} \star \gamma_j.
$$
Note that in particular
$$
\E \displaystyle \equiv \chi(X) \mathsf{P} \ \text{mod} \ q
$$
where $\chi(X)$ is the Euler characteristic of $X$.

In this paper we compute the quantum Euler class of all Fano nonsingular complete intersections of dimension at least $3$ in a projective space (see Theorem \ref{Main theorem} below). In particular, we prove a conjecture of Buch-Pandharipande, namely \cite[Conjecture 5.14]{BP}.

It is worth noting that although a priori the definition of $\E$ involves also the primitive cohomology of $X$, in our case of interest, this class actually lies in the restricted quantum cohomology ring $\QH^{\mathrm{res}}$ of $X$, that is the quantum cohomology ring coming from the projective space (see Proposition \ref{Graber} below for the exact definition of $\QH^{\mathrm{res}}$). This is a key reason we were able to obtain so explicit a formula for $\E$.

Finally, in \cite{BP} the quantum Euler class $\E$ of any variety $X$ is related via a very simple formula (see \cite[Theorem 1.4]{BP}) to the virtual Tevelev degrees of $X$, that is the virtual count of genus $g$ maps of fixed complex structure in a given curve class $\beta$ through $n$ general points of $X$. Exploiting their formula and our explicit expression of $\E$ for $X$ a Fano nonsingular complete intersections of dimension at least $3$, we are able to compute all the virtual Tevelev degrees of such varieties $X$ (see Theorem \ref{vTev} below).

\subsection[Preliminary results on complete intersections]{Preliminary results on complete intersections}
We now specialize to smooth complete intersections of dimension at least $3$. Let $X= V(f_1,...,f_L) \subset \P^{r+L}$ be a nonsingular complete intersection of dimension $r$. Assume for the rest of the paper that $r \geq 3$ and that for $i=1,...,L$,
$$
f_i \in H^0(\P^{r+L}, \mathcal{O}(m_i))
$$
where $m_i \geq 2$. 

Let $\m=(m_1,...,m_L)$ be the vector of degrees and, for $a,b \in \Z$ adopt the following notation:
$$
| \m|= \sum_{i=1}^L m_i,  \hspace{0.3cm}
\m^{a \m +b}= \prod_{i=1}^L m_i^{am_i+b}.
$$
\subsubsection{Cohomology of complete intersections}
Consider the map 
\begin{equation} \label{restriction map}
H^i(\P^{r+L}) \rightarrow H^i(X)
\end{equation}
induced by the inclusion $X \subset \P^{r+L}$. By the Lefschetz Hyperplane Theorem, this map is an isomorphism for all $i \leq 2r, i \neq r$ and is injective for $i=r$. Also, for $i=r$, we have a canonical decomposition 
$$
H^r(X)=H^r(X)^{\mathrm{prim}} \oplus H^r(X)^{\mathrm{res}}
$$
as a direct sum of the primitive cohomology and the restricted cohomology of degree $r$.

Explicitly 
$$
H^r(X)^{\mathrm{res}}= \mathrm{Im}( H^r(\P^{r+L}) \hookrightarrow H^r(X))
$$
and 
$$
H^r(X)^{\mathrm{prim}}= \mathrm{Ker}( \H \cup - : H^r(X) \rightarrow H^{r+2}(X))
$$
where $\H \in H^*(X)$ is the hyperplane class.

Note that
$$
\mathrm{dim} H^r(X)^{\mathrm{prim}}= (-1)^r(\chi(X)-(r+1))
$$
where $\chi(X)$ is the Euler characteristic of $X$.

\subsubsection{Quantum cohomology of Fano complete intersections}
From now on, we will further restrict our attention to the Fano case 
\begin{equation}\label{eqn:Fano}
|\m| \leq r+L.
\end{equation}
Since $ r \geq 3$, the map in Equation \ref{restriction map} is an isomorphism when $i=2$ and thus
$$
H_2(X)= \Q. \L
$$
where $\L \in H^*(X)$ is the class of a line in $X$. It follows that $\QH$ is a graded algebra over the polynomial ring $\Q [q]$ in one variable $q$ and as a $\Q[q]$-modules we have
$$
\QH = H^*(X) \otimes_{\Q} \Q[q].
$$
The degree of $q$ is equal to $2 d$ where
$$
d = r+L+1-| \m|.
$$
which is greater than $0$ by Equation \ref{eqn:Fano}.

Depending on the degree $|\m|$ of $X$, the ring $\QH$ satisfies the following magic relation (due to A. Givental):

\begin{itemize}
    \item if $|\m| \leq r+L-1$ we have: 
    \begin{equation}\label{eqn:magic1}
    \H^{\star (r+1)}= \m^{\m} q \H^{\star (|\m |-L)}
    \end{equation}
    \item if $|\m| = r+L$ we have:
    \begin{equation}\label{eqn:magic2}
    (\H+\m!q)^{\star (r+1)}= \m^{\m} q (\H + \m !q)^{\star r}
    \end{equation}
\end{itemize}
Some cases of Equation \ref{eqn:magic1} are proved in \cite{Bea}. A complete proof of both relations can be found in Givental's paper \cite{Gi}. There also are two very nice expositions of Givental's work, see \cite[Section 3.2]{P} and \cite[Corollary 4.4 and Corollary 4.19]{BDPP}. Relations  \ref{eqn:magic1} and \ref{eqn:magic2} will be essential for the object of this paper.  

\subsection{Statement of the main Theorem}
In Theorem \ref{Main theorem} below we will give explicit formulas for the quantum Euler class of any smooth Fano complete intersection $X\subseteq \P^{r+L}$ as above. Much of our work starts with the results in \cite{BP}, some of which we now recall for the reader's convenience.

\begin{proposition}[due to T. Graber] \label{Graber}
Let $\mathsf{R} = \mathrm{Span} \{1, \H,....,\H^r\} \subset H^*(X).$ Then,  $(\mathsf{R} \otimes_{\Q} \Q [q], \star)$ is a subring of $\QH$. 
\end{proposition}
\begin{proof}
This is \cite[Proposition 5.1]{BP}.
\end{proof}

This ring is denoted by $\QH^{\mathrm{res}}$. 

\begin{remark}\label{rmk:basis}
The elements $1,\H,...,\H^{\star r}$ form a basis of $\QH^{\mathrm{res}}$ as $\Q [q]$-module.
This follows from the fact that $\H^i= \H^{\star i}$ mod $q$ and $\H^{\star i} \in \QH^{\mathrm{res}}$ for $i=0,...,r$.
\end{remark}

Let $\E$ be the quantum Euler class of $X$.

\begin{lemma} \label{E is restricted}
We have $\E \in QH^*(X)^{\mathrm{res}}$.
\end{lemma}
\begin{proof}
See \cite[Proof of Proposition 5.5]{BP}.
\end{proof}

By Remark \ref{rmk:basis} and Lemma \ref{E is restricted}, we can uniquely write
$$
\E= \sum_{i = 0}^{\lfloor \frac{r}{d} \rfloor} \mathrm{Coeff}(\E,q^i \H^{\star (r-id)}) q^i \H^{\star (r-id)}.
$$
where $\mathrm{Coeff}(\E,q^i \H^{\star (r-id)}) \in \Q$. Our goal is to make this coefficients explicit.

\begin{remark}
We have 
$$
\mathrm{Coeff}(\E,\H^{\star r})= \m^{-1} \sum_j \int_X \gamma_j^\vee \cup \gamma_j = \m^{-1}\sum_j (-1)^{\mathrm{deg}(\gamma_j)}= \m^{-1} \chi (X).
$$
\end{remark}
The main result of the paper is the following:

\begin{theorem}[Main Theorem]\label{Main theorem}
The following equalities hold:
\begin{itemize}
    \item if $| \m | \leq r+L-1$ then 
    $$
    \E=\m^{-1} \chi(X) \H^{\star r}+(r+L+1-| \m|-\chi(X)) \m^{\m -1}q \H^{\star |\m|-L-1},
    $$
    \item if $| \m | = r+L$ then 
    $$
    \E=\m^{-1}\chi(X) \H^{\star r}+ \sum_{j=1}^r \m^{-1}(j- \chi(X) ){r \choose {j-1}} (\m !)^{j-1} \Bigg[\m^{\m}- \frac{\m !}{j} (r+1) \Bigg] q^j \H^{\star r-j}.
    $$
\end{itemize}
\end{theorem}

The case $| \m | \leq r+L-1$ in the theorem is exactly \cite[Conjecture 5.14]{BP} and is already shown to be true mod $q^2$ in \cite[Corollary 5.13]{BP}. The proof of this theorem is given in Section \ref{sec:Main theorem}.

\subsection{Application: virtual Tevelev degrees of Fano complete intersections}

\subsubsection{Virtual Tevelev degrees}\label{sec:def Tev degrees}
Let $X$ be a nonsingular, projective, algebraic variety over $\C$ of dimension $r$. Fix integers $g,n \geq 0$ satisfying the stability condition $2g-2+n >0$ and fix $\beta \in H_2(X, \Z)$ an effective curve class satisfying the condition
$$
\int_{\beta}c_1(T_X) >0.
$$
Let $\oM(X,\beta)$ be the moduli space of genus $g$, $n$-pointed stable maps to $X$ in class $\beta$ and assume that the dimensional constraint
$$
\mathrm{vdim}(\oM(X,\beta))=\mathrm{dim}(\oM \times X^n)
$$
holds. This is equivalent to
\begin{equation}\label{vTev dim constraint}
\int_{\beta}c_1(T_X)=r(n+g-1).
\end{equation}
Let
$$
\tau: \oM(X,\beta) \rightarrow \oM \times X^{n}
$$
be the canonical morphism obtained from the domain curve and the evaluation maps:
$$
\pi:\oM(X,\beta) \rightarrow \oM, \ \ \mathrm{ev}:\oM(X,\beta) \rightarrow X^n.
$$
Then the \textbf{virtual Tevelev degree} $\mathrm{vTev}^X_{g,n,\beta} \in \Q$ of $X$ is defined by the equality
$$
\tau_*[\oM(X, \beta)]^{\mathrm{vir}}= \mathrm{vTev}^X_{g,n,\beta} [\oM \times X^n] \in A^0(\oM \times X^n).
$$
Alternatively, denoting by $\Omega_{g,n,\beta}^X:H^*(X)^{\otimes n} \rightarrow H^*(\oM)$ the Gromov Witten class
$$
\Omega^X_{g,n,\beta}:(\alpha):=\pi_{*}(ev^*(\alpha) \cap [\oM (X,\beta)]^{\mathrm{vir}}),
$$
we have
$$
\mathrm{vTev}^X_{g,n,\beta} [\oM ]= \Omega_{g,n,\beta}^X( \mathsf{P}^{\otimes n}).
$$

Tevelev degrees were essentially introduced in \cite{T}, starting from which a lot of work has been done to study this degrees. In \cite{CPS} these degrees were formally defined and computed via Hurwitz theory for the case of $\P^1$; then, in \cite{FL}, using Schubert calculus the problem was posed and solved for the case of $\P^n$; in \cite{CL} a generalization of these degrees is presented for $\P^1$; in \cite{BP} a virtual perspective is adopted via Gromov-Witten theory; in \cite{LP} an equality between virtual and geometric Tevelev degrees is proven for certain Fano varieties and large degree curve classes and finally in \cite{L} geometric Tevelev degrees are computed for low degree hypersurfaces and large degree curve class via projective geometry. 

\subsubsection{Virtual Tevelev degrees of Fano complete intersections}
In this paper, we concern about exact computations of virtual Tevelev degrees of Fano complete intersections following the perspective presented in \cite{BP} and described above in Section \ref{sec:def Tev degrees}.

Let $X$ be a smooth Fano complete intersection in $\P^{r+L}$ of dimension $r \geq 3$ and vector of degrees $\m$. Writing $\beta=k \L$ with $k>0$, condition \ref{vTev dim constraint} becomes
$$
k=k[g,n]:=\frac{n+g-1}{d}r.
$$

For us, the main ingredient to computes $\vT$ will be the following result:

\begin{theorem}\label{thm:method of computation of vTev}
Suppose $k=k[g,n]$. Then
$$
\vT= \m^{1} \mathrm{Coeff}( \mathsf{P} ^{\star n} \star \E^{\star g},q^k \H^{\star r}).
$$
\end{theorem}
\begin{proof}
This is \cite[Theorem 1.4]{BP}.
\end{proof}

Before stating our theorem, we require a remark and some additional notation.

\begin{remark}
Given the form of Equation \ref{eqn:magic2}, when $|\m|=r+L$ it will be more convenient to use $1,(\H+\m! q),...,(\H+\m!q)^{\star r}$ instead of $1,\H,...,\H^{\star r}$ as a basis of $\QH$ as $\Q[q]$-module.
\end{remark}

\begin{definition}\label{defn:Pi and bi}
Define 
$$
\Q \ni P_i = 
\begin{cases}
\mathrm{Coeff}(\mathsf{P},q^i \H^{\star r-id}) &\text{when } | \m | \leq r+L-1,
\\  \mathrm{Coeff}(\mathsf{P},q^i (\H+\m!q)^{\star r-i}) &\text{when } | \m | = r+L,
\end{cases}
$$
for $i=0,...,\lfloor \frac{r}{d} \rfloor$ and 
$$
\Q \ni b_i = 
\begin{cases}
\mathrm{Coeff}(\mathsf{P}^{\star n} \star \E^{\star g},q^{i+k} \H^{\star r-id}) &\text{when } | \m | \leq r+L-1,
\\  \mathrm{Coeff}(\mathsf{P}^{\star n} \star \E^{\star g},q^{i+k} (\H+\m!q)^{\star r-i}) &\text{when } | \m | = r+L,
\end{cases}
$$
for $i=0,...,\lfloor \frac{r}{d} \rfloor$.
\end{definition}
Note that, by Theorem \ref{thm:method of computation of vTev}
$$
\vT= \m^1 b_0
$$
and that by Theorem \ref{Main theorem} the $b_i$'s are determined by the $P_i$'s.
\begin{definition}
Following \cite[Definition 5.15]{BP}, we define the \textbf{discrepancy} of $\mathsf{P}^{\star n} \star \E^{\star g}$ to be 
$$
\mathrm{Disc}(\mathsf{P}^{\star n} \star \E^{\star g})= \sum_{i=1}^{\lfloor \frac{r}{d} \rfloor} b_i \m^{-i \m +1}.
$$
\end{definition}

Putting all together we obtain explicit formulas for all virtual Tevelev degrees of $X$ (once all  the coefficients $P_i$ are known):

\begin{theorem}\label{vTev}
Suppose $k=k[g,n]$. Then, the virtual Tevelev degrees of $X$ are as follows:
\begin{itemize}
    \item if $| \m | \leq r+L-1$ then 
    $$
    \vT= \Bigg( \sum_{i=0}^{\lfloor \frac{r}{d} \rfloor} P_i \m^{-i \m} \Bigg)^n (r+L+1-| \m |)^g \m^{ k \m -g+1} - \mathrm{Disc}(\mathsf{P}^{\star n} \star \E^{\star g}),
    $$
    \item if $| \m | = r+L$ then 
    $$
    \vT= \Bigg( \sum_{i=0}^{r} P_i \m^{-i \m} \Bigg)^n \Bigg( 1- \m^{-r \m}(\m!)^r(r+1-\chi(X)) \Bigg)^g \m^{k \m -g+1}- \mathrm{Disc}( \mathsf{P}^{\star n} \star \E^{\star g}).
    $$
\end{itemize}
\end{theorem}
The case $|\m| \leq r+L-1$ already appears in \cite[Proposition 5.16]{BP} (where they assume that \cite[conjecture 5.14]{BP} holds for $X$), the case $|\m|=r+L$ is instead completely new. 

The last question would be if we can actually express the coefficients $P_i$ appearing in Theorem \ref{vTev} in a closed formula obtaining in this way a closed formula for the virtual Tevelev degrees.
\\ Partial results have been obtained in \cite{BP}, where they gave a complete answer in the following cases:
\begin{itemize}
\item for quadric hypersurfaces (see \cite[Theorem 1.5 and Example 2.4]{BP};
\item for low degree complete complete intersections $r > 2 | \m | -2L-2$ which are not quadrics (see \cite[Corollary 5.11 and Theorem 5.19]{BP});
\item for the border case $r=2 | \m | -2L-2$ (see \cite[Lemma 5.21 and Corollary 5.23]{BP}).
\end{itemize}
Here we will content ourselves with illustrating in Section \ref{sec: algorithm Pi} an algorithm that calculates all the coefficients $P_i$'s. It should be noted here that the method we will describe is more effective than the general result in \cite{ABPZ}, where they deal with Gromov-Witten invariants in all genera with arbitrary insertions.



\subsection*{Acknowledgments} I am especially thankful to my supervisor R. Pandharipande who introduced me to the topic of this paper, explained me some parts of the papers \cite{BP} and \cite{P} and read very carefully the first draft of this paper. I also would like to thank S. Molcho for reading and commenting the final draft of the paper,  C. Lian and J. Schmitt for several disscussions about Tevelev degrees and Younghan Bae for discussions about the content of Section \ref{sec: algorithm Pi}. I am supported by the grant SNF-200020-182181.

\section{A preliminary computation}

We start with expressing $\H^i$ for $i=1,...,r$ as a linear combination of $1,...,\H^{\star r}$ with coefficients in $\Q [q]$.

The following notation will be convenient. For $k \geq 0$ and $0 \leq j \leq r$ let
$$ 
\alpha^{k}_{r-j}:= \m ^{-1}\langle \H^{kd+j-1},\H^{r-j}\rangle^X_{0,k}= \m^{-1}\int_{[\overline{\mathcal{M}}_{0,2}(X,k \L)]^{\mathrm{vir}}} \mathrm{ev_1}^{*} \H^{kd+j-1} \cup \mathrm{ev}_2^{*} \H^{r-j}
$$
where $\L$ is the class of a line in $X$ and $\mathrm{ev}_i:\overline{\mathcal{M}}_{0,2}(X,k \L) \rightarrow X$ are the evaluation maps for $i=1,2$. Note the following symmetry:
$$
\alpha^{k}_{r-j}=\alpha^{k}_{kd+j-1}.
$$

\begin{proposition} \label{prop:explicit H^i}
Let $ 0 \leq i \leq r$. Then, for $1 \leq j \leq \lfloor \frac{i}{d} \rfloor$, we have 
\begin{equation*}
\begin{split}
\mathrm{Coeff}(\H^i,q^j & \H^{\star (i-jd)})= \\
& =\sum_{\ell: 1 \leq \ell \leq j} (-1)^{\ell } \sum_{\substack{(i_1,...,i_{\ell}) \in \Z_{\geq 1}: \\ i_1+...+i_{\ell}=j}} \sum_{\substack{(u_1,...,u_{\ell}) \in (\Z_{\geq 0})^{\times \ell}: \\ 0 \leq u_{\ell} \leq ... \leq u_1 \leq i-jd}} \prod_{a=1}^{\ell} i_a \alpha^{i_a}_{r-(j-i_1-...-i_a)d-u_a}.
\end{split}
\end{equation*}
\end{proposition}
\begin{proof}
We proceed by induction on $i$. 
\\When $i<d$ there is nothing to prove. When $i=d$, it must be $j=1$ and the right-hand side of the equation is just $\alpha^{1}_{r}$. To compute the left-hand side note that $\H^i= \H^{\star i}$ for $i<d$ and thus 
$$
\H^{\star d}=\H \star \H^{d-1}= \H^d +q \langle 
\H,\H^{d-1},\H^r\rangle^X_{0,1} \m^{-1}.
$$
Note that, since $\H \star \H^{d-1} \in \QH ^{\mathrm{res}}$, the primitive cohomology contributions in $\H \star \H^{d-1}$ is $0$. 
\\Use now the divisor equation in Gromov-Witten theory to obtain
$$
\langle \H,\H^{d-1},\H^r\rangle^X_{0,1} \m^{-1}= \alpha^{1}_r.
$$
Assume now that the Theorem is true for $i=d,...,t-1 < r$. 
\\Write 
$$
\H \star \H^{t-1}=\H^t+ \sum_{k=1}^{\lfloor \frac{t}{d} \rfloor} \langle\H,\H^{t-1},\H^{kd+r-t}\rangle^{X}_{0,k}\m^{-1}q^k \H^{t-kd}.
$$
where again the primitive cohomology contributions in $\H \star \H^{t-1}$ is $0$ and by the divisor equation 
$$
\m^{-1}\langle \H,\H^{t-1},\H^{kd+r-t} \rangle^{X}_{0,k}=k \alpha^{k}_{r-(t-kd)}.
$$
Note now that, by induction, we know how to write the $\H^{t-kd}$ and $\H^{t-1}$ in terms of $1,\H,...,\H^{\star r}$.
Putting all together we obtain for $1 \leq j \leq \lfloor \frac{t}{d} \rfloor$ 
\begin{equation*}
\begin{split}
\mathrm{Coeff}(\H^t,q^j \H^{\star (t-jd)}) =&\mathrm{Coeff}(\H^{t-1}, q^j \H^{\star (t-1-jd)})-\sum_{k=1}^j k \alpha^{k}_{r-(t-kd)} \mathrm{Coeff}(\H^{t-kd},q^{j-k} \H^{\star (t-jd)})
\\ =& \sum_{\ell=1}^{j} (-1)^{\ell } \sum_{\substack{i_1+...+i_{\ell}=j, \\ 0 \leq u_{\ell} \leq ... \leq u_1 \leq t-1-jd}} \prod_{a=1}^\ell  i_a \alpha^{i_a}_{r-(j-i_1-...-i_a)d-u_a} \\
-&\sum_{k=1}^j k \alpha^{k}_{r-(t-kd)} \Bigg( \sum_{\ell=0}^{j-k}(-1)^{\ell } \sum_{\substack{i_1+...+i_{\ell}=j-k, \\ 0 \leq u_\ell \leq ... \leq u_1 \leq t-jd}} \prod_{a=1}^\ell  i_a \alpha^{i_a}_{r-(j-k-i_1-...-i_a)d-u_a} \Bigg)
\end{split}
\end{equation*}
and, since 
$$
r-(t-kd)=r-(j-k)d - (t-jd),
$$
this is exactly the right-hand side appearing in the statement above with $i=t$.
\end{proof}
\begin{corollary} 
For $ 1 \leq j \leq \lfloor \frac{r}{d} \rfloor$ we have 
\begin{equation}\label{eqn:Compuatation of Hr+1}
\mathrm{Coeff}(\H^{\star r+1},q^j \H^{\star r+1-jd})=\sum_{\ell: 1 \leq \ell \leq j} (-1)^{\ell+1} \sum_{\substack{(i_1,...,i_{\ell}) \in \Z_{\geq 1}: \\ i_1+...+i_{\ell}=j}} \sum_{\substack{(u_1,...,u_{\ell}) \in (\Z_{\geq 0})^{\times \ell}: \\ 0 \leq u_{\ell} \leq ... \leq u_1 \leq r+1-jd}} \prod_{a=1}^{\ell} i_a \alpha^{i_a}_{r-(j-i_1-...-i_a)d-u_a}.
\end{equation}
 
\end{corollary}
\begin{proof}
Proceeding as in the proof of Proposition \ref{prop:explicit H^i}, we write $\H^{\star r+1}=\H \star \H^{\star r}$ and
$$
\H^{\star r}= \H^r - \sum_{j=1}^{\lfloor \frac{r}{d} \rfloor} \mathrm{Coeff}(\H^r,q^j \H^{\star r-jd}) q^j \H^{\star r-jd}.
$$
Since 
$$
\H \star \H^r= \sum_{k=1}^{\lfloor \frac{r+1}{d} \rfloor} k \alpha^{k}_{kd-1} q^k \H^{r-kd+1},
$$
we have 
\begin{equation*}
\begin{split}
\mathrm{Coeff}(\H^{\star (r+1)},q^j \H^{\star (r+1-jd)}) &=-\mathrm{Coeff}(\H^{r}, q^j \H^{\star (r-jd)})
\\&+\sum_{k=1}^j k \alpha^{k}_{r-(r+1-kd)} \mathrm{Coeff}(\H^{r+1-kd},q^{j-k} \H^{\star r+1-jd})\\
&=\sum_{1 \leq \ell \leq j}(-1)^{\ell+1} \sum_{\substack{i_1+...+i_{\ell}=j, \\ 0 \leq u_\ell \leq ... \leq u_1 \leq r-jd}} \prod_{a=1}^\ell i_a \alpha^{i_a}_{r-(j-i_1-...-i_a)d-u_a}\\
&- \sum_{k=1}^j k \alpha^{k}_{kd-1} \sum_{1 \leq \ell \leq j-k}(-1)^{\ell+1} \sum_{\substack{i_1+...+i_{\ell}=j-k, \\ 0 \leq u_\ell \leq ... \leq u_1 \leq r+1-jd}}\prod_{a=1}^\ell i_a \alpha^{i_a}_{r-(j-k-i_1-...-i_a)d-u_a}
\end{split}
\end{equation*}
for $j=1,...,\lfloor \frac{r}{d} \rfloor$. Finally, since $kd-1=r-(j-k)d-(r+1-jd)$, we are done.
\end{proof}
Using the previous Corollary and Equations \ref{eqn:magic1} and \ref{eqn:magic2} we get interesting equalities.

\section{Proof of the main theorem}\label{sec:Main theorem}

\subsection{Plan of the Proof }
In this subsection we explain how the proof of Theorem \ref{Main theorem} goes. 

Define 
$$
\Gamma:= \sum_{j: \gamma_j \in H^r(X)^{\mathrm{prim}}} \gamma_j^\vee \star \gamma_j
$$
and 
$$
\E':=\m^{-1} \sum_{i=0}^r \H^i \star \H^{r-i}
$$
Then 
$$
\E=\Gamma + \E'
$$
and so by Lemma \ref{E is restricted}, we see that $\Gamma \in QH^*(X)^{\mathrm{res}}$.

Using relations \ref{eqn:magic1} and \ref{eqn:magic2}, the proof of Theorem \ref{Main theorem} becomes an easy algebraic count (done in Section \ref{sec:Computation of E}) once we know the following two propositions.

\begin{proposition} \label{coeff Gamma}
For $j=1,...,{\lfloor \frac{r}{d} \rfloor}$ we have
$$
\mathrm{Coeff}(\Gamma,q^j \H^{\star r-jd})=\m^{-1}(r+1- \chi(X)) \mathrm{Coeff}(\H^{\star r+1},q^j \H^{\star r+1-jd}).
$$
\end{proposition}
The proof is presented in Section \ref{sec:computation Gamma}.

\begin{proposition}\label{coeff E'}
For $j=1,...,{\lfloor \frac{r}{d} \rfloor}$ we have
$$
\mathrm{Coeff}(\E',q^j \H^{\star r-jd})= - \m ^{-1}(r-jd+1) \mathrm{Coeff}(\H^{\star r+1},q^j \H^{\star r+1-jd}).
$$
\end{proposition}

The proof is presented in Section \ref{sec:computation E'}.

We remark here that the way we prove Proposition \ref{coeff E'} is purely algebraic. It would be interesting to find a more conceptual explanation for this equality.

\subsection{Computation of \texorpdfstring{$\Gamma$}{Gamma}}\label{sec:computation Gamma}
The proof of Proposition \ref{coeff Gamma} relies on the following preliminary lemma which is very similar to \cite[Lemma 5.2]{BP}.

\begin{lemma}\label{product with H}
Let $\Lambda \in QH^*(X)^{\mathrm{res}}$ be a degree $2r$ class such that 
$$
\Lambda=a \H^r \ \mathrm{mod} \ q \ \mathrm{and} \ \H  \star \Lambda =0
$$
where $a \in \Q $. Then 
$$
\mathrm{Coeff}(\Lambda, q^j \H^{\star r-jd})=-a \mathrm{Coeff}(\H^{\star r+1},q^j \H^{r+1-jd})
$$
for $i=1,...,\lfloor \frac{r}{d} \rfloor$.
\end{lemma}
\begin{proof}
Write 
$$
\Lambda= a \H^{\star r} + \sum_{i=1}^{\lfloor \frac{r}{d} \rfloor} \mathrm{Coeff}(\Lambda, q^i \H^{\star r-id}) q^i \H^{\star r-id}.
$$
Then we have 
$$
0= \H \star \Lambda = a \H^{\star r+1} + \sum_{i=1}^{\lfloor \frac{r}{d} \rfloor} \mathrm{Coeff}(\Lambda, q^i \H^{\star r-id}) q^i \H^{\star r+1-id}
$$
from which we obtain the lemma.
\end{proof}

\begin{proof}[Proof of Proposition \ref{coeff Gamma}]
By \cite[Corollary 5.3]{BP} (the same proof of that corollary also applies when $|\m|=r+L$), we have $\H \star \Gamma =0$. Moreover
$$
\Gamma= \sum_{j: \gamma_j \in H^r(X)^{\mathrm{Prim}} } \gamma_j^\vee \cup \gamma_j= \mathrm{dim} (H^r(X)^{\mathrm{prim}}) (-1)^r \m^{-1} \H^r \ \ \mathrm{mod} \ q.
$$
Note that we $\mathrm{dim} (H^r(X)^{\mathrm{prim}})=(-1)^r(\chi(X)-r-1)$. The proposition now follows from an application of Lemma \ref{product with H}.

\end{proof}

\subsection{Computation of \texorpdfstring{$\E'$}{E'}}\label{sec:computation E'}
In this subsection we prove Proposition \ref{coeff E'} by showing the following equivalent result:
\begin{lemma}\label{lem:computation E'}
For $j=1,...,\lfloor \frac{r}{d} \rfloor$ we have
$$
\sum_{i=0}^r \mathrm{Coeff}(\H^i \star \H^{r-i}, q^j \H^{\star (r-jd)})= -(r-jd+1) \ \mathrm{times} \ \mathrm{the} \ \mathrm{RHS} \  \mathrm{of} \  \mathrm{Equation} \ \ref{eqn:Compuatation of Hr+1} .
$$
\end{lemma}
\begin{proof}
For $0 \leq i \leq r$ and $1 \leq j \leq \lfloor \frac{r}{d} \rfloor$ we have
$$
\mathrm{Coeff}(\H^i \star \H^{r-i}, q^j \H^{\star (r-jd)}) 
= \sum_{(h,s) \in \Z_{\geq 0}^{\times 2}:h+s=j} \mathrm{Coeff}(\H^i,q^h \H^{\star (i-hd)})\mathrm{Coeff}(\H^{r-i},q^s \H^{\star (r-i-sd)})
$$
and for each $(h,s)$ as above such that $hd \leq i$ and $r-i \geq sd$, the product $$
\mathrm{Coeff}(\H^i,q^h \H^{\star (i-hd)})\mathrm{Coeff}(\H^{r-i},q^s \H^{\star (r-i-sd)})
$$
is equal to 
\begin{equation} \label{product1}
\sum_{\substack{1 \leq w \leq h \\ 1 \leq z \leq s}}(-1)^{w + z} \sum_{\substack{y_1+...+y_w=h \\ x_1+...+x_z=s}}  \sum_{\substack{0 \leq p_w \leq ... \leq p_1 \leq i-hd \\ 0 \leq v_z \leq ...\leq v_1 \leq r-i-sd}} \prod_{a=1}^w i_a \alpha^{i_a}_{r-(h-y_1-...-y_a)d-p_a} \prod_{b=1}^z x_b \alpha^{x_b}_{r-(s-x_1-...-x_b)d-v_b}.
\end{equation}
Note that it could be $h=0$ or $s=0$ (but not $h=s=0$ being $h+s=j>0$).
Observe the symmetry
$$
\alpha^{x_b}_{r-(s-x_1-...-x_b)d-v_b}=\alpha^{x_b}_{(s-x_1-...-x_{b-1})d+v_b-1},
$$
and that 
$$
(s-x_1-...-x_{b-1})d+v_b-1=r-[(j-x_b-...-x_z)d +r-v_b+1-jd]
$$
where $r-v_b+1-jd$ varies in $[i-hd+1,r+1-jd]$ for $0 \leq v_b \leq r-i-sd$. Therefore we can rewrite the quantity in Equation \ref{product1} as
\begin{equation}\label{product2}
\sum_{\substack{1 \leq w \leq h \\ 1 \leq z \leq s}}(-1)^{w + z} \sum_{\substack{y_1+...+y_w=h \\ x_1+...+x_z=s}}  \sum_{\substack{0 \leq p_w \leq ... \leq p_1 \leq i-hd \\ i-hd+1 \leq v_1 \leq ...\leq v_z \leq r+1-jd}} \prod_{a=1}^w i_a \alpha^{i_a}_{r-(h-y_1-...-y_a)d-p_a} \prod_{b=1}^z x_b \alpha^{x_b}_{r-(j-x_b-...-x_z)d- v_b}.
\end{equation}
Note that the quantity $i-hd$ appearing in Equation \ref{product2} under the third summation symbol varies in $[0, r-jd]$ and not in $[0,r]$ (if $i-hd >r-jd$, then $r-i<(j-h)d=sd$ and so $\mathrm{Coeff}(\H^{r-i},q^s \H^{\star (r-i-sd)})=0$).

Fix $j \in \{1,...,\lfloor \frac{r}{d} \rfloor \}$
and let $\ell,(i_1,...,i_\ell)$ and $(u_1,...,u_\ell)$ be such that 
$$
0 \leq \ell \leq j, \ i_1+...+i_\ell=j \ \mathrm{and} \  0\leq u_\ell \leq ... \leq u_1 \leq r+1-jd.
$$  
We want to count how many times the term
\begin{equation} \label{terms}
(-1)^\ell \prod_{a=1}^\ell i_a \alpha^{i_a}_{r-(j-i_1-...-i_a)d-u_a}
\end{equation}
appears in 
$$
b_j :=\sum_{i=0}^r \mathrm{Coeff}(\H^i \star \H^{r-i}, q^j \H^{\star (r-jd)})=\sum_{i=0}^r \sum_{h+s=j} \mathrm{Coeff}(\H^i,q^h \H^{\star (i-hd)})\mathrm{Coeff}(\H^{r-i},q^s \H^{\star (r-i-sd)}).
$$
\\First of all we observe that it must be 
$w+z= \ell$
and
$
(x_z,...,x_1,y_1,...,y_w)=(i_1,...,i_\ell).
$
Moreover, fixed any integer $g \in [0, r-jd]$, if $i-hd=g$ then in Equation \ref{product2} it must be 
$$
z=\mathrm{min}\{ f: u_f \leq g \}-1 \ \mathrm{and} \ w=\ell-z
$$
where if $\{ f: u_f \leq g \}= \emptyset$, we set $z= \ell$ and $w=\ell$.
\\Therefore
$$
v_z=u_1,...,v_1=u_z,p_1=u_{z+1},...,p_w=u_\ell
$$
and 
$$
x_z=i_1,...,x_1=i_z,y_1=i_{z+1},...,y_w=i_\ell
$$
and finally
$$
h=y_1+...+y_w \ \mathrm{and} \ s=x_1+...+x_z.
$$
This means that the term in Equation \ref{terms} appears in $b_j$ once for every $g \in [0,r-jd]$, and thus a total of $r-jd+1$ times. This concludes the proof of the lemma.
\end{proof}

\subsection{Computation of \texorpdfstring{$\E$}{E}}\label{sec:Computation of E}
We finally prove Theorem \ref{Main theorem}. We will distinguish two cases.

\vspace{8pt}
\noindent \textbf{ $\bullet$ Case $|\m| \leq r+L-1$.}
\vspace{8pt}

By Relation \ref{eqn:magic1}, in this case we have
$$
\mathrm{Coeff}(\H^{\star r+1},q^j \H^{\star r+1-jd})= 0 \ \mathrm{for} \ j>1 \ \mathrm{and} \
\mathrm{Coeff}(\H^{\star r+1},q \H^{\star r+1-d})=\m^{\m}.
$$
Therefore, by Propositions \ref{coeff Gamma} and \ref{coeff E'}, we have
$$
\mathrm{Coeff}(\E,q^j \H^{\star r-jd})=0 \ \mathrm{for} \ j>1 \ \mathrm{and} \
\mathrm{Coeff}(\E,q \H^{\star r-d})=\m^{m-1}(r+L+1-|\m|-\chi(X)).
$$
This is what we wanted to prove.

\vspace{8pt}
\noindent \textbf{ $\bullet$ Case $|\m| = r+L$.}
\vspace{8pt}

Note that in this case $d=1$. Relation \ref{eqn:magic2}, can be rewritten as 
$$
\mathrm{Coeff}(\H^{\star r+1},q^j \H^{\star r+1-j})=  {r \choose {j-1}} (\m !)^{j-1} \Bigg[\m^{\m}- \frac{\m !}{j} (r+1) \Bigg]
$$
for $j=1,...,r+1$. Therefore for $j=1,...,r$ we have
\begin{align*}
\mathrm{Coeff}(\E,q^j \H^{\star r-j})&=\m^{-1}(j- \chi(X) ) \mathrm{Coeff}(\H^{\star r+1},q^j\H^{\star r+1-j})\\
&=\m^{-1}(j- \chi(X) ){r \choose {j-1}} (\m !)^{j-1} \Bigg[\m^{\m}- \frac{\m !}{j} (r+1) \Bigg].
\end{align*}
This concludes the proof.

\section{Virtual Tevelev degrees }
We now apply our computations to prove Theorem \ref{vTev}. We distinguish two cases again.

\vspace{8pt}
\noindent \textbf{ $\bullet$ Case $|\m| \leq r+L-1$.}
\vspace{8pt}

This case follows from \cite[Proposition 5.16]{BP} and Theorem \ref{Main theorem} above.

\vspace{8pt}
\noindent \textbf{ $\bullet$ Case $|\m| = r+L$.}
\vspace{8pt}

The first step is to express $\E$ in terms of the basis $1,\H+\m!q,...,(\H+\m!q)^{\star r}$.
This will use the following simple combinatorial lemma.
\begin{lemma}
For $j=2,...,r$ the following two equalities hold:
\begin{equation}\label{eqn:comb1}
\sum_{i=1}^j {r \choose {i-1}}{{r-i} \choose {j-i}}(-1)^{j-i}=1
\end{equation}
and 
\begin{equation}\label{eqn:comb2}
\sum_{i=1}^j i{r \choose {i-1}}{{r-i} \choose {j-i}}(-1)^{j-i}=r+1.
\end{equation}
\end{lemma}
\begin{proof}
The proof is left to the reader.
\end{proof}

\begin{lemma}\label{lem:E in the new basis}
We have 
\begin{align*}
\E= \m^{-1} \chi(X) (\H+\m!q)^{\star r}+[\m^{-1}(r+1- \chi(X))( \m^{\m}- \m !)- \m^{\m-1}r ]q(\H+\m!q)^{\star r-1}&\\
+ \sum_{j=2}^r [\m^{-1}(\m !)^{j-1}(r+1-\chi(X))( \m^\m - \m!)]q^j(\H+\m!q)^{\star r-j}&. 
\end{align*}
\end{lemma}
\begin{proof}
This is an algebraic check substituting
$$
\H=(\H +\m ! q) - \m ! q
$$
in the expression of $\E$ found in Theorem \ref{Main theorem}.
\\Here we will deal with $\mathrm{Coeff}(\E,q^j(\H+\m!q)^{\star r-j})$ for $j=2,...,r$. The cases $j=0,1$ are instead left to the reader. 
\\Using Theorem \ref{Main theorem}, we see that for $j=2,...,r$ the coefficient $\mathrm{Coeff}(\E,q^j(\H+\m!q)^{\star r-j})$ is equal to
\begin{align*}
    \m^{-1} \chi(X) {r \choose j}& (-1)^j (\m !)^j \\ & +\sum_{i=1}^j \m^{-1} (i-\chi(X)) {r \choose {i-1}} (\m!)^{i-1}\Bigg[\m^\m -\m ! \frac{(r+1)}{i}\Bigg] {r-i \choose j-i}(-1)^{j-i}(\m!)^{j-i} 
\end{align*}
which we now rewrite as a sum of four terms. The first one is
$$
\m^{-1}\chi(X)( \m !)^{j-1} \m! \Bigg[ (-1)^j {r \choose j} + \sum_{i=1}^j (-1)^{j-i} {r \choose {i-1}} {{r-i} \choose {j-i}}\frac{r+1}{i} \Bigg] 
=\m^{-1}\chi(X)( \m !)^{j-1} \m!
$$
where we used Equation \ref{eqn:comb1}. The second one is
$$
-\m^{-1}\chi(X)( \m !)^{j-1} \m^\m \sum_{i=1}^j {r \choose i-1} {r-i \choose j-i}(-1)^{j-i}=-\m^{-1}\chi(X)( \m !)^{j-1} \m^\m
$$
where we used Equation \ref{eqn:comb1}. The third term is
$$
-\m^{-1} (\m !)^{j-1} \m! \sum_{i=1}^j {r \choose i-1} {r-i \choose j-i}(-1)^{j-i}(r+1)=-\m^{-1} (\m !)^{j-1} \m! (r+1)
$$
where we used again Equation \ref{eqn:comb1}. Finally the last term is
$$
\m^{-1}(\m !)^{j-1} \m^\m \sum_{i=1}^j i {r \choose i-1}{r-i \choose j-i}(-1)^{j-i}=\m^{-1}(\m !)^{j-1} \m^\m (r+1)
$$
where instead we used Equation \ref{eqn:comb2}.
\\Summing everything up we obtain the desired conclusion.
\end{proof}

Although the full expression of $\E$ might be a bit complicated, the product $(\H+\m!q)^{\star r} \star \E$ is quite simple.

\begin{corollary}\label{cor: E is simple}
We have 
$$
(\H+\m!q)^{\star r} \star \E=[\m^{-1}-\m^{-r \m -1}(\m !)^r (r+1-\chi(X))] (\H+ \m!q)^{\star 2r}.
$$
\end{corollary}
\begin{proof}
Use $r$ times Equality \ref{eqn:magic2}.
\end{proof}
We can now finish the proof of Theorem \ref{vTev}.

\begin{proof}[Proof of Theorem \ref{vTev} when $|\m|=r+L$] From Definition \ref{defn:Pi and bi} and Equation \ref{eqn:magic2}, we see that
\begin{equation}\label{eqn:comp1}
\mathsf{P}^{\star n} \star \E^{\star g} \star (\H+\m! q)^{\star r}= \Bigg( \sum_{i=0}^r b_i \m^{-(k+i)\m} \Bigg) (\H+q \m!)^{\star r+rg+nr}.
\end{equation}
Using Definition \ref{defn:Pi and bi} and Equation \ref{eqn:magic2}, we also have
$$
\mathsf{P}^{\star n} \star (\H+\m! q)^{\star r}=\Bigg( \sum_{i=0}^r P_i \m^{-i\m} \Bigg)^n (\H+q \m!)^{\star nr+r}.
$$
and so by Corollary \ref{cor: E is simple}
\begin{equation}\label{eqn:comp2}
\mathsf{P}^{\star n} \star \E^{\star g} \star (\H+\m! q)^{\star r}= \Bigg( \sum_{i=0}^r P_i \m^{-i\m} \Bigg)^n \Bigg(\m^{-1}-\m^{-r \m -1}(\m !)^r (r+1-\chi(X)) \Bigg)^g  (\H+\m!q)^{\star r+gr+nr}.
\end{equation}
The theorem follows by comparing Equation \ref{eqn:comp1} and Equation \ref{eqn:comp2}.
\end{proof}

\section{An algorithm for the calculation of the coefficients $P_i$} \label{sec: algorithm Pi}

In this final section we propose a method to compute the coefficients $P_i$ appearing in Definition \ref{defn:Pi and bi}. In this way, up to implementing the algorithm with a computer, all the virtual Tevelev degrees of $X$ can be explicitly calculated. 

It is possible that our is known to the experts, but we preferred to include it anyway for completeness.

\subsection{Recursion for genus $0$ two-pointed Hyperplane Gromov-Witten invariants}\label{sec:recursion P_i}

Proposition \ref{prop:explicit H^i} reduces the computation of the $P_i$s to the computation of genus $0$ two-pointed Hyperplane Gromov Witten invariants of $X$. These invariants satisfies a recursion involving more general integrals which we now recall.

\subsubsection{The recursion}

For $g \geq 0$, $k>0$ and $n>0$ the \textbf{gravitational descendant invariants} of $X$ are defined by:
$$
\langle \tau_{a_1}(\gamma_1),...,\tau_{a_n}(\gamma_n)\rangle^{X}_{g,k}:= \int_{[\oM(X,k \L)]^{\mathrm{vir}}} \mathrm{ev}_1^*(\gamma_1) \cup \psi_1^{a_1} \cup ... \cup \mathrm{ev}_n^*(\gamma_n) \cup \psi_1^{a_n}
$$
where $\gamma_1,...,\gamma_n \in H^*(X)$ and $\psi_i=c_1(\mathbb{L}_i) \in H^2(\oM(X,k \L))$ is the first Chern class of the cotangent line
$$
{\mathbb{L}_i}_{ \big| [f:(C,p_1,...,p_n) \rightarrow X]}= (T_{p_i}C)^{\vee}
$$
for $i=1,...,n$. 

We start with a monodromy result.
\begin{lemma}\label{lem:monodromy argument}
For any $\gamma \in H^*(X)^{\mathrm{prim}}$, $\gamma_1,...,\gamma_n \in H^*(X)^{\mathrm{res}}$ (with $n \geq 0$), $a_1,...,a_n \in \Z_{\geq 0}$ and $k>0$ we have
$$
\langle \tau_{a_1}(\gamma_1),...,\tau_{a_n}(\gamma_n),\gamma \rangle^X_{0,k}=0.
$$
\end{lemma}
\begin{proof}
The proof is a monodromy argument. Let 
$$
U \subset \prod_{i=1}^L \P(H^0(\P^{r+L},\mathcal{O}(m_i)))
$$
be the open subscheme parametrizing smooth complete intersection in $\P^{r+L}$ of dimension $r$ and degree $\m$. Call
$
V=V^{\mathrm{prim}} \oplus V^{\mathrm{res}}
$
where $V^{\mathrm{prim}}=H^*(X)^{\mathrm{prim}} \otimes_{\Q} \R$ and $V^{\mathrm{res}}=H^*(X)^{\mathrm{res}} \otimes_{\Q} \R$, and
$$
\rho:\pi_1(U,u) \rightarrow \mathrm{Aut}(V).
$$
the monodromy homomorphism (here $u \in U$ is the point corresponding to $X$). The homomorphism $\rho$ preserves the decomposition 
$
V=V^{\mathrm{prim}} \oplus V^{\mathrm{res}}
$
and actually its invariant subspace is exactly $V^{\mathrm{res}}$. Let $G \subset \mathrm{GL}(V^{\mathrm{prim}})$ be the algebraic monodromy group defined as the Zariski closure of the image of $\pi_1(U,u) \rightarrow \mathrm{Aut}(V^{\mathrm{prim}})$. The lemma will follow from the following two standard facts:
\begin{itemize}
\item the invariance under deformations of $X$ in Gromov Witten theory tells us that for any $\alpha \in \pi_1(U,u)$ we have:
$$
\langle \tau_{a_1}(\alpha.\gamma_1),...,\tau_{a_n}(\alpha. \gamma_n),\alpha. \gamma \rangle ^X_{0,k}=\langle \tau_{a_1}(\gamma_1),...,\tau_{a_n}(\gamma_n),\gamma \rangle ^X_{0,k};
$$
\item the intersection form $Q$ on $V^{\mathrm{prim}}$ is preserved by the monodromy action. When $r$ is odd, $Q$ is a non-degenerate skew-symmetric bilinear form, it follows that in this case $\mathrm{dim}(V^{\mathrm{prim}})$ is even and that $G \subseteq \mathrm{Sp}(V^{\mathrm{prim}})$. When instead $r$ is even, $Q$ is a non-degenerate symmetric bilinear form and we have $G \subseteq \mathrm{O}(V^{\mathrm{prim}})$. Since for us $r \geq 3$, by \cite[Theorem 4.4.1]{Del} (see also \cite[Proposition 4.2]{ABPZ}), the previous inclusions are actually equalities except for the case when $r$ is even and $\m=(2,2)$. In this latter case, $\mathrm{dim}(V^{\mathrm{prim}})=r+3$ and $G$ is the Weil group $\mathcal{W}$ of $D_{r+3}$. 
\end{itemize}
Since $- \mathrm{Id} \in \mathrm{Sp}(V^{\mathrm{prim}})$ and $- \mathrm{Id} \in \mathrm{O}(V^{\mathrm{prim}})$ the proof is complete in all cases except for the case $r$ even and $\m=(2,2)$. In this case, note that if $L:V^{\mathrm{prim}} \rightarrow V^{\mathrm{prim}}$ is any $\R$-linear map invariant under $\mathcal{W}$ then it must be $L=0$ (reason: if $\Phi \subset V^{\mathrm{prim}}$ is the root system corresponding to $D_{r+3}$ then for all $v \in \Phi$ the reflection $r_v$ along the hyperplane $v^{\perp}$ lies in $\mathcal{W}$ and sends $v$ to $-v$, thus $L(v)=L(-v)=-L(v)$, from which $L(v)=0$. Since $\mathrm{Span}_{\R}(\Phi)=V^{\mathrm{prim}}$ we are done). To conclude the proof of the lemma, apply this observation with $L=\langle \tau_{a_1}(\gamma_1),...,\tau_{a_n}(\gamma_n), - \rangle ^X_{0,k}$.

\end{proof}
\begin{proposition}\label{prop:recursion} Let $i,a \geq 0$ and $j,k >0$ be integers satisfying
$$
i+j+a= \mathrm{vdim}( \overline{\mathcal{M}}_{0,2}(X,k \L) ).
$$
Then we have
\begin{align*}
\langle \tau_a(\H^i),\H^j\ \rangle^X_{0,k}&=\langle \tau_a(\H^{i+1}),\H^{j-1}\rangle^X_{0,k}+k\langle \tau_{a+1}(\H^i),\H^{j-1} \rangle^X_{0,k} \\
&-\sum_{\ell=1}^{k-1} \m^{-1} \ell \langle \tau_a(\H^i),\H^{\ell d+r-1-i-a}\rangle^X_{0,\ell}\langle \H^{j-1},\H^{(k-\ell)d+r-j}\rangle^X_{0,k-\ell}.
\end{align*}
\end{proposition}
\begin{proof}
An application of \cite[Corollary 1]{LeeP} and the splitting principle in Gromov Witten theory yields
\begin{align*}
\langle \tau_a(\H^i),\H^j \rangle^X_{0,k}&=\langle \tau_a(\H^{i+1}),\H^{j-1}\rangle^X_{0,k}+k\langle \tau_{a+1}(\H^i),\H^{j-1} \rangle ^X_{0,k} \\
&-\sum_{\ell=1}^{k-1}  \sum_{j=0}^N \ell \langle \tau_a(\H^i),\gamma_j^{\vee} \rangle ^X_{0,\ell} \langle \H^{j-1},\gamma_j \rangle^X_{0,k-\ell}
\end{align*}
where as always $\{\gamma_j\}_{j=0}^N$ is any homogeneous basis of $H^*(X)$ with $\gamma_0=1$ and $\gamma_{N}=\mathsf{P}$. Finally, apply Lemma \ref{lem:monodromy argument} to conclude the proof.
\end{proof}

\subsubsection{The base case }

Consider the recursion of Proposition \ref{prop:recursion}. In each two pointed Gromov Witten integral on the right-hand side, either the quantity $j$ decreased or the quantity $k$ decreased (when compared to those appearing in the left-hand side). Note also that when $k=1$, the recursion becomes simply
$$
\langle \tau_a(\H^i),\H^j \rangle ^X_{0,1}=\langle \tau_a(\H^{i+1}),\H^{j-1} \rangle^X_{0,1}+\langle \tau_{a+1}(\H^i),\H^{j-1} \rangle^X_{0,1}.
$$
So, when $k=1$, $k$ stabilizes while $j$ keeps going down. 
It follows that the recursion completely determines all the integrals 
$$
\langle \tau_a(\H^i),\H^j \rangle ^X_{0,k} \ \mathrm{for} \ a,i,j \geq 0 \ \mathrm{such}\  \mathrm{that} \  a+i+j=\mathrm{vdim}( \overline{\mathcal{M}}_{0,2}(X,k \L) )
$$
once the integrals $\langle \tau_a(\H^i),1 \rangle ^X_{0,k}$ are given for all $a,i \geq 0$ and $k >0 $ . These last invariants are indeed known as the next proposition shows.

\begin{proposition}
Let $a,i \geq 0$ and $k>0$ be integers such that 
$$
a+i=\mathrm{vdim}( \overline{\mathcal{M}}_{0,2}(X,k \L)).
$$
Then
\begin{itemize}
    \item for $| \m | \leq r+L-1$ and $i=0,...,r$ we have
    $$
    \langle \tau_{r+kd-1-i}(\H^i),1 \rangle^X_{0,k}=\mathrm{Coeff}\Bigg(\frac{\prod_{j=1}^L \prod_{\ell=0}^{km_j}(m_j x+\ell)}{\prod_{\ell=1}^k(x+\ell)^{r+L+1}},x^{r+L-i}\Bigg);
    $$
    \item for $| \m | = r+L$ and $i=0,...,r$ we have
    $$
    \langle \tau_{r+k-1-i}(\H^i),1 \rangle ^X_{0,k}=\sum_{h=0}^k \frac{(- \m !)^{k-h}}{(k-h)!}\mathrm{Coeff}\Bigg(\frac{\prod_{j=1}^L \prod_{\ell=0}^{hm_j}(m_j x+\ell)}{\prod_{\ell=1}^h(x+\ell)^{r+L+1}},x^{r+L-i}\Bigg)
    $$
    \end{itemize}
    where in both cases the coefficient of $x^{r+L-i}$ is meant to be the coefficient of the Taylor expansion in $x$ at $0$.
\end{proposition}
\begin{proof}
This is just a way of rephrasing \cite[Theorem 4.2 and Theorem 4.17]{BDPP}. Note that in \cite[Theorem 4.17]{BDPP} there is typo: in their notation, their index $m$ in the product appearing in the numerator should range from $0$ to $dl_j$, instead of from $1$ to $d$.
\end{proof}

\vspace{8pt}

\noindent Departement Mathematik, ETH Z\"urich\\
\noindent alessio.cela@math.ethz.ch

\end{document}